\documentclass[12pt,a4paper]{amsart}
\usepackage{amsmath}
\usepackage{mathrsfs}
\usepackage{amsxtra}
\usepackage{graphics,color,graphicx}
\usepackage{amssymb, amsfonts,amscd,verbatim,hyperref,cleveref}
\usepackage{xspace}
\usepackage{tikz}
\usetikzlibrary{matrix,calc,decorations.pathmorphing,shapes,arrows}
\newtheorem{theorem}{Theorem}[section]
\newtheorem{proposition}[theorem]{Proposition}
\newtheorem{corollary}[theorem]{Corollary}
\newtheorem{lemma}[theorem]{Lemma}

\numberwithin{equation}{section}
\theoremstyle{definition}

\newtheorem{example}[theorem]{Example}

\newcommand{\pA}{\mathcal{A}}
\newcommand{\pB}{\mathcal{B}}

\newcommand{\pF}{\mathcal{F}}

\newcommand{\pH}{\mathcal{H}}

\newcommand{\pJ}{\mathcal{J}}
\newcommand{\pK}{\mathcal{K}}
\newcommand{\pL}{\mathcal{L}}
\newcommand{\pM}{\mathcal{M}}
\newcommand{\pN}{\mathcal{N}}
\newcommand{\pR}{\mathcal{R}}
\newcommand{\pS}{\mathcal{S}}

\newcommand{\eC}{\mathscr{C}}

\newcommand{\eM}{\mathscr{M}}

\newcommand{\eX}{\mathscr{X}}
\newcommand{\eY}{\mathscr{Y}}

\newcommand{\bC}{\mathbb{C}}

\newcommand{\bN}{\mathbb{N}}

\newcommand{\rad}{{\pR}}
\begin{document}
\title{Hyperinvariant subspaces for sets of polynomially compact operators}
\author{J.~Bra\v ci\v c}
\address{Faculty of Natural Sciences and Engineering,
University of Ljubljana,
A\v sker\v ceva cesta 12
SI-1000 Ljubljana,
Slovenija }
\email{janko.bracic@ntf.uni-lj.si}
\author{M.~Kandi\' c}
\address{Faculty of Mathematics and Physics,
University of Ljubljana,
Jadranska ulica 19,
SI-1000 Ljubljana,
Slovenija }
\email{marko.kandic@fmf.uni-lj.si}

\address{Institute of Mathematics, Physics and Mechanics,
Jadranska ulica 19,
SI-1000 Ljubljana,
Slovenija }
\email{marko.kandic@fmf.uni-lj.si}
\keywords{Polynomially compact operator; hyperinvariant subspace}
\subjclass[2020]{Primary: 47A15, 47B07 Secondary: 47L10, 47B10}
%
%
\begin{abstract}
We prove the existence of a non-trivial hyperinvariant subspace for several sets of polynomially compact operators.
The main results of the paper are:

(i) a non-trivial norm closed algebra $\pA\subseteq \pB(\eX)$ which consists of polynomially compact quasinilpotent operators has a non-trivial hyperinvariant subspace;

(ii) if there exists a non-zero compact operator in the norm closure of the algebra generated by an operator band $\pS$,
then $\pS$ has a non-trivial hyperinvariant subspace.
\end{abstract}
\maketitle
%
\section {Introduction} \label{section1}

Let $\eX$ be a complex Banach space. Denote by $\pB(\eX)$ the algebra of all bounded linear operators on $\eX$.
A closed subspace $\eM\subseteq \eX$ is said to be invariant for an operator $T\in \pB(\eX)$ if $T\eM\subseteq \eM$.
Let $\pS\subseteq \pB(\eX)$ be a non-empty set of operators. Then $\eM$ is an invariant subspace of $\pS$ if it is
invariant for every operator in $\pS$. If $\eM$ is invariant for every operator in $\pS$ and for every operator in the
commutant $\pS'=\{ T\in \pB(\eX);\; TS=ST\,\text{for every}\, S\in \pS\}$, then it is a hyperinvariant subspace of
$\pS$. Of course, the trivial subspaces $\{ 0\}$ and $\eX$ are (hyper)invariant for any set of operators. We are interested
in the existence of non-trivial invariant and hyperinvariant subspaces. The problem of existence of invariant and hyperinvariant subspaces for a given operator or a non-empty set of operators is an extensively studied topic in operator theory. The problem is solved in the finite-dimensional setting by Burnside's theorem (see \cite[Theorem 1.2.2]{RR}). In the context of infinite-dimensional Banach spaces, the problem is open for reflexive Banach spaces, in particular, for the infinite-dimensional separable Hilbert space. However, there are some Banach spaces for which we know either that every operator
has a non-trivial invariant subspace or that there exist operators without it. For instance, Argyros and Haydon
\cite{AH} have proved the existence of an infinite-dimensional Banach spaces $\eX$ such that every operator in
$\pB(\eX)$ is of the form $\lambda I+K$, where $I$ is the identity operator and $K$ is compact. It follows, by the celebrated von Neumann-Aronszajn-Smith theorem \cite{AS} and Lomonosov's theorem \cite{Lom}, that any operator in $\pB(\eX)$ has a non-trivial invariant subspace and any non-scalar operator in $\pB(\eX)$ has a non-trivial hyperinvariant subspace.
On the other hand, several examples of Banach spaces (including $\ell_1$) with operators without a non-trivial
invariant subspace are known (see \cite{Enf,Rea} for the first examples and \cite{GR} for a general approach to
Read’s type constructions of operators without non-trivial invariant closed subspaces).

With the von Neumann-Aronszajn-Smith theorem and Lomonosov's theorem in mind it is not a surprise that
suitable compactness conditions imply existence of non-trivial invariant and hyperinvariant subspaces for
different classes of operators and sets of operators. For instance, Shulman \cite[Theorem 2]{Shu} proved that an
algebra of operators whose radical contains a non-zero compact operator has a non-trivial hyperinvariant subspace.
Turovskii \cite[Corollary 5]{Tur} extended this result to semigroups of quasinilpotent operators. Another type of results
are those related to triangularizability of a set of operators. Recall that a non-empty set $\pS\subseteq \pB(\eX)$
is triangularizable if there exists a chain $\eC$ which is maximal as a chain of subspaces of $\eX$ and every subspace in $\eC$ is invariant for all operators in $\pS$. Every commutative set of compact operators is triangularizable (see \cite[Theorem 7.2.1]{RR}). Konvalinka \cite[Corollary 2.6]{Kon} has extended this result by
showing that a commuting family of polynomially compact operators is triangularizable. Another result in this direction,
obtained by the second author \cite{Kan},
says that for a norm closed subalgebra $\pA\subseteq \pB(\eX)$ of power compact operators the following assertions are equivalent: (a) $\pA$ is triangularizable; (b) the Jacobson radical $\pR(\pA)$ consists precisely of quasinilpotent
operators in $\pA$; (c) the quotient algebra $\pA/\pR(\pA)$ is commutative.

The aim of this paper is to consider the problem of existence of a non-trivial hyperinvariant subspace for
sets of polynomially compact operators. For instance, we prove (\Cref{theo03}) that a non-trivial norm closed algebra $\pA\subseteq \pB(\eX)$
which consists of polynomially compact quasinilpotent operators has a non-trivial hyperinvariant subspace.
Another result (\Cref{theo05}) which we mention here is related to operator bands, that is, to semigroups of idempotent operators. It says that an operator band $\pS$ has a non-trivial hyperinvariant subspace
if there exists a non-zero compact operator in the norm closure of the algebra generated by $\pS$.

%
\section{Preliminaries} \label{section2}
\subsection{Notation}
Let $\eX$ be a non-trivial complex Banach space. Since the results proved in this paper are either trivial or well-known
when $\eX$ is finite-dimensional we always assume that $\dim(\eX)=\infty$.
Let $\pB(\eX)$ denote the Banach algebra of all bounded linear operators on $\eX$ and let $\pK(\eX)\subseteq \pB(\eX)$ be the ideal of compact operators. The identity operator is denoted by $I$ and an operator is said to be scalar
if it is a scalar multiple of $I$.
The norm closure of a set $\pS\subseteq \pB(\eX)$ is denoted by $\overline{\pS}$.

For two operators $S_1, S_2\in \pB(\eX)$, we denote their commutator $S_1S_2-S_2S_1$ by $[S_1,S_2]$.
A non-empty set $\pS\subseteq \pB(\eX)$ is commutative if any two operators from $\pS$ commute, that is,
 $[S_1,S_2]=0$ for all $S_1,S_2\in \pS$. Similarly, $\pS$ is essentially commutative if $[S_1,S_2]\in \pK(\eX)$.
For an arbitrary
non-empty subset $\pS\subseteq \pB(\eX)$ the commutant of $\pS$ is $\pS'=\{ T\in\pB(\eX);\; [T,S]=0,\, \forall\, S\in \pS\}$. It is clear that $\pS'$ is a closed subalgebra of $\pB(\eX)$. If $\pS$ is commutative, then
$\pS\subseteq \pS'$.

\subsection{Invariant subspaces.}
A non-empty subset $\eM\subseteq \eX$ is a subspace if it is a closed linear manifold.
It is said that a subspace $\eM$ of $\eX$ is invariant for operator $T$ if $T\eM\subseteq \eM$.
An invariant subspace $\eM$ is non-trivial if $\{ 0\}\ne \pM\ne \eX$.
A non-empty set $\pS\subseteq \pB(\eX)$ is reducible if there exists a non-trivial subspace $\eM\subseteq \eX$  which is invariant for every $T\in \pS$. If there exists a chain $\eC$ which is maximal as a chain of subspaces of $\eX$ and every subspace in $\eC$ is invariant for all operators in $\pS$, then $\pS$ is said to be triangularizable.

If a subspace $\eM$ is invariant for every operator $T$ in a set $\pS$ and in its commutant $\pS'$, then $\eM$ is said to be a hyperinvariant subspace for $\pS$.

Next theorem (see Assertion in the end of \cite{Lom}) is one of the deepest results in the theory of invariant subspaces.

\begin{theorem}[Lomonosov] \label{theo06}
Every non-scalar operator which commutes with a non-zero compact operator has a non-trivial hyperinvariant subspace.
\end{theorem}

\noindent
The proof of \Cref{theo06} relies on the following useful lemma.

\begin{lemma}[Lomonosov's Lemma] \label{lem02}
Let $\pA\subseteq \pB(\eX)$ be an algebra. If $\pA$ does not have a non-trivial invariant subspace, then for
every non-zero compact operator $K\in \pB(\eX)$ there exist an operator $A\in \pA$ and a non-zero vector $x\in \eX$
such that $AKx=x$.
\end{lemma}

\subsection{Spectral radius.}
The spectrum of an operator $T\in \pB(\eX)$ is denoted by $\sigma(T)$ and the spectral radius of $T$ is
$\rho(T)=\max\{ |z|;\; z\in \sigma(T)\}$. By the spectral radius formula (Gelfand's formula),
$\rho(T)=\lim\limits_{n\to \infty} \| T\|^{\frac{1}{n}}$.
An operator $T\in \pB(\eX)$ is quasinilpotent if $\rho(T)=0$. A compact quasinilpotent operator is called a Volterra operator.

Let $\pF$ be a non-empty set of operators in $\pB(\eX)$. For each $n\in \bN$, let $\pF^{(n)}=\{T_1\cdots T_n;\; T_1,\ldots, T_n\in \pF\}$. By $\|\pF\|=\sup\{\|T\|;\; T\in \pF\}$ we denote the joint norm of $\pF$ and by $\rho(\pF)$ we denote the joint spectral radius of $\pF$ defined as
$$\rho(\pF)=\limsup_{n\to\infty} \|\pF^{(n)}\|^{\frac 1n}.$$

A subalgebra $\pA\subseteq \pB(\eX)$ is said to be finitely quasinilpotent if $\rho(\pF)=0$ for every finite subset $\pF$ of $\pA$. By \cite[Theorem 1]{Shu}, every subalgebra $\pA\subseteq \pB(\eX)$  of Volterra operators is finitely
quasinilpotent.

\subsection{Polynomially compact operators.}
An operator $T\in \pB(\eX)$ is polynomially compact if there exists a non-zero complex polynomial $p$ such that $p(T)$ is a compact operator. In particular, algebraic operators (nilpotents, idempotents etc.) are polynomially compact.
Hence, $T$ is polynomially compact if and only if $\pi(T)$, where
$\pi\colon \pB(\eX)\to \pB(\eX)/\pK(\eX)$ is the quotient projection, is an algebraic element in the Calkin algebra
$\pB(\eX)/\pK(\eX)$. If $T^n$ is compact for some $n\in \bN$, then $T$ is said to be power compact. For a polynomially compact operator $T$, there exists a unique monic polynomial $m_T$ of the smallest degree such that $m_T(T)$ is a compact operator. The polynomial $m_T$ is called the minimal polynomial of $T$.
The following is the structure theorem for polynomially compact operators proved by Gilfeather \cite[Theorem 1]{Gil}.

\begin{theorem} \label{theo02}
Let $T\in \pB(\eX)$ be a polynomially compact operator with minimal polynomial $m_T(z)=(z-\lambda_1)^{n_1}\cdots
(z-\lambda_k)^{n_k}$. Then there exist invariant subspaces $\eX_1,\ldots, \eX_k$ for $T$ such that $\eX=\eX_1 \oplus \cdots \oplus \eX_k$ and $T=T_1\oplus \cdots \oplus T_k$, where $T_i$ is the restriction of $T$ to $\eX_i$. The operators
$(T_j-\lambda_j I_j)^{n_j}$ are all compact.

The spectrum of $T$ consists of countably many points with $\{ \lambda_1,\ldots,\lambda_k\}$ as the only possible
limit points and such that all but possibly $\{ \lambda_1,\ldots,\lambda_k\}$ are eigenvalues with finite dimensional
generalized eigenspaces. Each point $\lambda_j$ $(j=1, \ldots,k)$ is either the limit of eigenvalues of $T$ or else
$\eX_j$ is infinite dimensional and $T_j-\lambda_j I_j$ is a quasinilpotent operator on $\eX_j$.
\end{theorem}

\begin{corollary} \label{cor01}
A quasinilpotent operator $T\in \pB(\eX)$ is polynomially compact if and only if it is power compact.
\end{corollary}

\begin{proof}
It is clear that every power compact operator is polynomially compact. On the other hand, if $T$ is a polynomially compact and quasinilpotent, then $\sigma(T)=\{0\}$ and therefore $m_T(z)=z^n$ for a positive integer $n$, by
\Cref{theo02}, that is, $T$ is power compact.
\end{proof}

\subsection{Algebras and ideals.}
For a non-empty set of operators $\pS\subseteq \pB(\eX)$, let $\pA(\pS)$ be the subalgebra of $\pB(\eX)$ generated by $\pS$ and let $\pA_1(\pS)$ be the subalgebra of $\pB(\eX)$ generated by $\pS$ and $I$.
By $\pH(\pS)$ we denote the algebra which is generated by $\pS$ and its commutant $\pS'$. We will call it
the {\em hyperalgebra} of $\pS$. If $\pS$ is a semigroup, then an operator
$T\in \pB(\eX)$ is in $\pH(\pS)$ if and only if there exist $n\in \bN$ and operators $S_1, \ldots, S_n\in \pS$ and
$T_0, T_1, \ldots, T_n\in \pS'$ such that
$$T=T_0+S_1 T_1+\cdots+ S_n T_n=T_0+T_1 S_1+\cdots+ T_n S_n .$$
Here we used the fact that $I\in \pS'$.
Since $\pS\subseteq \pH(\pS)$ and $\pA(\pS)$ is the smallest algebra which contains $\pS$ we have
$\pA(\pS)\subseteq \pH(\pS)$. On the other hand, it is obvious that $\pS'=\pA(\pS)'$ and therefore
$\pA(\pS)'\subseteq \pH(\pS)$. We conclude that the hyperalgebra of $\pS$ is generated by $\pA(\pS)$ and
$\pA(\pS)'$, that is, $\pH(\pS)=\pH\bigl(\pA(\pS)\bigr)$.

If $\pM$ and $\pN$ are non-empty subsets of $\pB(\eX)$, then let
$\pM+\pN=\{ M+N;\; M\in \pM, N\in \pN\}$ and let
$\pM\pN$ be the set of all finite
sums $M_1 N_1+\cdots+M_k N_k$, where $M_i\in \pM$ and $N_i\in \pN$ for each $i=1, \ldots, k$.
Hence, if $\pA$ is a subalgebra of $\pB(\eX)$, then its hyperalgebra is
$\pH(\pA)=\pA'+\pA\pA'=\pA'+\pA'\pA$.

If $\pJ$ is an ideal in an algebra $\pA\subseteq \pB(\eX)$, then we will denote by $\pJ_{_{\pH}}$ the ideal in $\pH(\pA)$ generated by $\pJ$.

\begin{lemma} \label{lem01}
Let $\pA\subseteq \pB(\eX)$ be an algebra. If $\pJ\triangleleft\pA$, then $\pJ_{_{\pH}}=\pA'\pJ=\pJ\pA'$.
\end{lemma}

\begin{proof}
From equalities $\pJ_{_{\pH}}=\pJ+\pH(\pA)\pJ+\pJ\pH(\pA)+\pH(\pA)\pJ\pH(\pA)$ and $\pH(\pA)=\pA'+\pA'\pA$ we  conclude
\begin{align*}
\pJ_{_{\pH}}&=\pJ+\pH(\pA)\pJ+\pJ\pH(\pA)+\pH(\pA)\pJ\pH(\pA)\\
&=\pJ+\mathcal (\pA'+\pA'\pA)\pJ+\pJ\mathcal (\pA'+\pA'\pA)+\mathcal (\pA'+\pA'\pA)\pJ\mathcal (\pA'+\pA'\pA)\\
&\subseteq \pJ+\pA'\pJ \subseteq \pA' \pJ
\end{align*}
as $\pA'$ contains the identity operator.
On the other hand, since we also have $\pA'\pJ\subseteq \pJ_{_{\pH}}$, we obtain the equality
$\pJ_{_{\pH}}=\pA'\pJ$.
\end{proof}

An algebra $\pA$ over an arbitrary field is said to be a nil-algebra if every element of $\pA$ is nilpotent.
A nil-algebra $\pA$ is of bounded nil-index if there exists a positive integer $n$ such that $x^n=0$ for each $x\in \pA$.
If there exists $n\in\bN$ such that $a_1\cdots a_n=0$ for all $a_1,\ldots,a_n\in \pA$, then $\pA$ is said to be a nilpotent algebra. By the celebrated Nagata-Higman theorem (see \cite{Nag} and \cite{Hig}), every nil-algebra of bounded nil-index is nilpotent.

\begin{lemma} \label{lem03}
Let $\pA\subseteq \pB(\eX)$ be an algebra. An ideal $\pJ\triangleleft \pA$ is nilpotent if and only if $\pJ_{_{\pH}}\triangleleft \pH(\pA)$ is nilpotent. The nilpotency indices of $\pJ$ and $\pJ_{_{\pH}}$ are equal.
\end{lemma}

\begin{proof}
Since $\pA'$ and $\pJ$ commute, an easy induction shows that for each $n\in \bN$ we have
$(\pA'\pJ)^n=\pA' \pJ^n$. Hence, if $\pJ^n=\{0\}$, then $(\pA'\pJ)^n=\{0\}$, as well. If $(\pA'\pJ)^n=\{0\}$, then $\pJ^n \subseteq \pA'\pJ^n=(\pA'\pJ)^n=\{0\}$ yields that $\pJ^n=\{0\}$.
\end{proof}

If $\pJ$ is a nil-ideal of bounded nil-index, then $\pJ$ is nilpotent by the Nagata-Higman theorem. This immediately implies that
$\pJ_{_{\pH}}$ is nilpotent.
In particular, if $\pA$ is a nilpotent algebra, then $\pA\pA'$ is a nilpotent ideal in the hyperalgebra $\pH(\pA)$.

\section{Hyperinvariant subspaces of algebras of polynomially compact operators} \label{section3}

The simplest polynomially compact operators which are not necessary compact are algebraic operators,
in particular nilpotent operators.
Hadwin et al. \cite[Corollary 4.2]{HNRRR} proved that a norm closed algebra of nilpotent operators on the separable
infinite-dimensional complex Hilbert space is triangularizable. The following proposition shows that a norm closed algebra of nilpotent operators on an arbitrary complex Banach space has a non-trivial hyperinvariant subspace.

\begin{proposition} \label{prop01}
If a norm closed subalgebra $\pA\subseteq \pB(\eX)$ consists of nilpotent operators, then its hyperalgebra $\pH(\pA)$
is reducible.
\end{proposition}

\begin{proof}
By Grabiner's Theorem \cite{Grab},  $\pA$ is a nilpotent algebra, that is,  there exists $n\in \bN$ such that an arbitrary product of at least $n$ operators from $\pA$ is the zero operator. Let $n_0$ be the smallest positive integer with this property. Since $\pA\ne \{ 0\}$ we have $n_0>1.$ There exist operators $A_1,\ldots,A_{n_0-1}\in \pA$ such that
$A_0:=A_1,\ldots,A_{n_0-1}\ne 0$. Note that $A_0T=TA_0=0$ for every operator $T\in \pA$, that is, $\pA\subseteq
(A_0)'$, where $(A_0)'$ is the commutant of $A_0$. It is clear that $\pA'\subseteq (A_0)'$. Hence, $\pH(\pA)\subseteq (A_0)'$.
Since $A_0\ne 0$ the kernel $\ker(A_0)$ is a non-trivial subspace of $\eX$ and it is hyperinvariant for $A_0$.
It follows that $\ker(A_0)$ is a non-trivial hyperinvariant subspace for $\pA$, that is, $\pH(\pA)$ is reducible.
\end{proof}

The proof of \Cref{prop01} reveals the following result.

\begin{corollary} \label{cor02}
If $\{ 0\}\ne \pA\subseteq \pB(\eX)$ is a nilpotent algebra, then it has a non-trivial hyperinvariant subspace.
\end{corollary}

Hadwin et al. have constructed a semi-simple algebra of nilpotent operators on a separable Hilbert space
such that for each pair $(x,y)$ of vectors, where $x\ne 0$, there exists an operator $A\in \pA$ such that $Ax=y$
(see \cite[Section 4]{HNRRR}). Hence, \Cref{prop01} (and consequently \Cref{theo03} below) does not hold, in general, for non-closed algebras. However, as the following example shows there are simple special cases
when a not necessarily closed algebra generated by a set of nilpotents has a non-trivial hyperinvariant subspace.

\begin{example} \label{ex01}
The idea for this example is from \cite[Theorem 4.3]{HNRRR} where quadratic nilpotent operators are considered.
Choose and fix $\lambda \in \bC$. Let $\pS\subseteq \pB(\eX)$ be a non-empty set of
operators such that $A^2=\lambda A$ for every $A\in \pS$.
If $\pS$ is a linear manifold and contains a non-scalar operator, then $\pA(\pS)$ has a non-trivial hyperinvariant
subspace. To see this, choose a non-scalar operator $A \in \pS$.
Then its kernel $\ker(A)$ is a non-trivial hyperinvariant subspace for $A$. Since $\pA(\pS)'\subseteq (A)'$ we see that
$\ker(A)$ is invariant for every operator from $\pA(\pS)'$.  Let $B\in \pS$ be arbitrary. It follows
from $(A+B)^2=\lambda(A+B)$ that $AB=-BA$. Hence, $ABx=-BAx=0$ for every $x\in \ker(A)$, that is, $\ker(A)$ is
invariant for every $B\in \pS$. We conclude that $\ker(A)$ is invariant for every operator from the hyperalgebra
 $\pH(\pS)$.
\end{example}

\begin{proposition} \label{prop04}
Let $\pA\subseteq \pB(\eX)$ be a non-trivial subalgebra. If there exists a non-trivial nilpotent ideal $\pJ\triangleleft \pA$, then $\pA$ has a non-trivial hyperinvariant subspace.
\end{proposition}

\begin{proof}
Let $\pJ$ be a non-trivial nilpotent ideal in $\pA(\pS)$. Then, by \Cref{lem03}, the ideal $\pJ_{_{\pH}}$
which is generated by $\pJ$ in the hyperalgebra $\pH(\pA)$ is nilpotent, as well. By \Cref{cor02} the ideal $\pJ_{_{\pH}}$ has a non-trivial hyperinvariant subspace, in particular, it is reducible. By \cite[ Lemma 7.4.6]{RR},
$\pH(\pA)$ is reducible, as well.
\end{proof}

The following two theorems show that an algebra $\pA$ of polynomially compact operators has a non-trivial hyperinvariant subspace if $\pA$ satisfies some additional condition. For instance, as we already mentioned, the key assumption in
\Cref{theo03} is that the involved algebra is norm closed.

\begin{theorem}\label{theo03}
Let $\{ 0\}\ne \pA\subseteq \pB(\eX)$ be a norm closed algebra.
If every operator in $\pA$ is quasinilpotent and polynomially compact, then $\pA$ has a non-trivial hyperinvariant subspace.
\end{theorem}

\begin {proof}
Note first that each operator in $\pA$ is power compact, by \Cref{cor01}. If each operator in $\pA$ is nilpotent, then
$\pH(\pA)$ is reducible, by \Cref{prop01}.
Assume therefore that there exists an operator $T\in \pA$ which is not nilpotent. Since $T$ is power compact there exists $m\in \bN$ such that $K=T^m\ne 0$  is compact.
Let $\pJ$ be the two-sided ideal in $\pA_1$ generated by $K$. It is clear that $K\in \pJ\subseteq \pA$. Hence, $\pJ$ is an algebra
of Volterra operators. By \cite[Theorem 1]{Shu}, $\pJ$ is finitely quasinilpotent.
Let $A=B_0+\sum_{i=1}^{n} A_iB_i$, where $ A_i\in \pA$ and $B_i\in \pA'$, be an arbitrary operator in $\pH(\pA)$
and let $\pM=\{K,A_1K,\ldots,A_nK\}$. Since $\pM$ is a finite subset of $\pJ$ it is a quasinilpotent set, that is,
$\rho(\pM)=0$. Let
$\pN=\{B_0, B_1,\ldots,B_n\}\subseteq \pA'$. Since $B_0, B_1,\ldots,B_n$ commute with operators from $\pM$
we have $\rho(AK)\leq (n+1)\rho(\pM) \rho(\pN)=0$,
by \cite[Lemma 1]{Shu}. This shows that operator $AK$ is quasinilpotent for each $A\in \pH(\pA)$.
It follows, by \Cref{lem02}, that the hyperalgebra $\pH(\pA)$ is reducible.
 \end {proof}

\begin{theorem} \label{theo04}
Let $\pS\subseteq \pB(\eX)$ be an essentially commutative set of polynomially compact operators which contains at least one non-scalar operator. If $\pS$ is triangularizable, then $\pA(\pS)$ has a non-trivial closed hyperinvariant subspace.
\end{theorem}

\begin{proof}
It is obvious that every closed subspace of $\eX$ which is invariant for $\pS$ is invariant for $\pA(\pS)$, as well.
Hence, $\pA(\pS)$ is triangularizable. Since $\pS$ consists of polynomially compact operators, the same holds for
the algebra $\pA(\pS)$, by \cite[Theorem 1.5]{Kon}. Denote by $\pi\bigl(\pA(\pS)\bigr)$ and $\pi(\pS)$ the image
of $\pA(\pS)$ and $\pS$, respectively, in the Calkin algebra $\pB(\eX)/\pK(\eX)$. Since $\pi\bigl(\pA(\pS)\bigr)$ is
the subalgebra of the Calkin algebra generated by the commutative set $\pi(\pS)$ we see that $\pi\bigl(\pA(\pS)\bigr)$
is commutative,  as well. It follows that $\pA(\pS)$ is an essentially commutative subalgebra of $\pB(\eX)$. By
\cite[Theorem 3.5]{Kan}, for an algebra of essentially commuting polynomially compact operators triangularizability of
$\pA(\pS)$ is equivalent to the fact that each commutator $[S,T]$, where $S, T\in \pA(\pS)$, is a quasinilpotent operator.

Suppose that there exist operators $S, T\in \pA(\pS)$ such that the commutator $K:=[S,T]$ is non-zero.
Let $\pJ$ be the ideal in $\pA(\pS)$ generated by $K$. Of course, $\pJ$ consists of compact operators and, by Ringrose's Theorem (see \cite[Theorem 7.2.3]{RR}), every operator in $\pJ$ is quasinilpotent. Hence, $\pJ$ is an algebra of Volterra operators. Let $\pJ_{_{\pH}}$ be the ideal
in the hyperalgebra $\pH(\pS)$ generated by $\pJ$. It is clear that $\pJ_{_{\pH}}\subseteq \pK(\eX)$.
We claim that $\pJ_{_{\pH}}$ consists of quasinilpotent operators, as well. Choose an arbitrary operator $A\in \pJ_{_{\pH}}$. By \Cref{lem01} there exist operators $J_1,\ldots,J_n\in \pJ$ and operators $B_1,\ldots,B_n\in \pA(\pS)'$ such that
$A=\sum_{i=1}^{n} J_i B_i$. By \cite[Lemma 1]{Shu}, $\rho(A)\leq n \rho(\{J_1,\ldots,J_n\})  \rho (\{B_1,\ldots, B_n\})$. Since each finite subset of an algebra of Volterra operators is quasinilpotent, by \cite[Theorem 1]{Shu}, we have $\rho(\{J_1,\ldots,J_n\})=0$ and consequently $\rho(A)=0$. We have proved that $\pJ_{_{\pH}}$ is a non-trivial
ideal of Volterra operators. By \cite[Theorem 2]{Shu}, the hyperalgebra $\pH(\pS)$ is reducible.

Suppose now that $\pA(\pS)$ is commutative. Hence $\pA(\pS)\subseteq \pA(\pS)'$. Let $T$ be any non-scalar polynomially compact operator in $\pA(\pS)$ and let $m_T$ be its minimal polynomial. Hence, $m_T(T)$ is a compact operator. If $m(T)$ is a non-zero compact operator, then it has a non-trivial hyperinvariant subspace $\eY$, by Lomonosov's Theorem. Since $\pA(\pS)\subseteq \pA(\pS)'\subseteq (m_T(T))'$ subspace $\eY$ is invariant for every
operator in $\pA(\pS)$ and $\pA(\pS)'$. Thus, $\pH(\pS)$ is reducible. On the other hand, if
$m_T(T)=0$, then $T$ is a non-scalar algebraic operator. Hence, for every $\lambda\in \sigma(T)$, the kernel
$\ker(T-\lambda I)$ is a closed non-trivial hyperinvariant subspace for $T$, and consequently, for $\pA(\pS)$ as
$\pA(\pS)\subseteq \pA(\pS)'\subseteq (T)'$.
\end{proof}

\section{Hyperinvariant subspaces for operator bands} \label{section4}

An operator band on a Banach space $\eX$ is a (multiplicative) semigroup $\pS\subseteq \pB(\eX)$ of idempotents,
that is, $S^2=S$ for each $S\in \pS$. A linear span of an operator band $\pS$ is the algebra $\pA(\pS)$ called a band
algebra. We will denote by $\pN(\pS)$ the set of all nilpotent operators in the band algebra $\pA(\pS)$. By
\cite[Theorem 5.2]{LMMR}, $\pN(\pS)$ is the linear span of $[\pS,\pS]$, that is,
$\pN(\pS)=\{ \sum_{i=1}^{n}[A_i,B_i];\; n\in \bN,\, A_i,B_i\in \pA(\pS)\}$.
By \cite[Corollary 5.5]{LMMR}, the set $\pN(\pS)$ of nilpotent operators in $\pA(\pS)$ coincides with the Jacobson radical $\rad\bigl(\pA(\pS)\bigr)$ of $\pA(\pS)$.

\begin{proposition} \label{prop02}
Let $\{ 0\}\ne \pS\subseteq \pB(\eX)$ be an operator band. If $\pN\subseteq \pN(\pS)$ is a non-empty set, then
the ideal $\pJ_{_{\pH}}$ in the hyperalgebra $\pH(\pS)$ generated by $\pN$  is a nil-ideal.
\end{proposition}

\begin{proof}
Denote by $\pJ$ the ideal in $\pA(\pS)$ generated by $\pN$. We have already mentioned that $\rad\bigl(\pA(\pS)\bigr)=\pN(\pS)$. Since the radical is an ideal, we have that
$\pJ \subseteq \rad\bigl(\pA(\pS)\bigr)$. Hence, $\pJ$ is a nil-ideal in $\pA(\pS)$.

It is clear that $\pN$ and ideal $\pJ$ generate the same ideal $\pJ_{_{\pH}}$ in the hyperalgebra $\pH(\pS)$.
Let $A\in \pJ_{_{\pH}}$ be arbitrary. By \Cref{lem01}, there exist $n\in\bN$, $J_i\in \pJ$ and $B_i\in \pA(\pS)'$ such
that $A=\sum_{i=1}^n J_i B_i$. Let $\pF$ be a finite subset of $\pS$ such that $J_1,\ldots,J_n$ are contained in the algebra $\pA(\pF)$ generated by $\pF$. Note that such finite sets exist because each $J_i$ is of the form
$p_i(S_{1}^{(i)}, \ldots, S_{k_i}^{(i)})$, where $p_i$ is a polynomial of $k_i$ non-commuting variables and
$S_{1}^{(i)}, \ldots, S_{k_i}^{(i)}$ $(i=1, \ldots, n)$ are idempotents from $\pS$. Denote by $\pS_0$ the operator
band ganerated by $\pF$. Thus, $\pA(\pF)\subseteq \pA(\pS_0)$. Since $\pF$ is finite,  the operator band $\pS_0$ is finite, as well, by the Green-Rees theorem (see \cite[Theorem 9.3.11]{RR}). Now we apply \cite[Theorem 9.3.15]{RR} which says that there exists a finite chain $\{0\}=\eX_0\subseteq \eX_1\subseteq \cdots\subseteq \eX_{m}=\eX$
of invariant subspaces for $\pS_0$ such that for each $E\in \pS_0$ the operator induced by $E$ on $\eX_i/\eX_{i-1}$ is either zero or the identity operator. This implies that operator induced by $J_i$ on $\eX_j/\eX_{j-1}$ is a scalar multiple of the identity operator.
Since each $J_i$ is nilpotent, every operator $J_i$ induces the zero operator on $\eX_j/\eX_{j-1}$. Hence, for each $i$ and each $j$ we have $J_i(\eX_j)\subseteq \eX_{j-1}$. This implies that an arbitrary product of length at least $m$ with letters from $\{J_1,\ldots,J_n\}$ is zero. Now it is obvious that $A^{m}=\left(\sum_{i=1}^n J_i B_i\right)^m=0$
as $J_i$ and $B_j$ commute.
\end{proof}

\begin{proposition} \label{prop03}
Let $\{ 0\}\ne \pS\subseteq \pB(\eX)$ be an operator band. If $K\in \overline{\pA(\pS)}$ is a compact operator,
then for each operator $A\in \overline{\pH(\pS)}$ the commutator $[K,A]$ is in the Jacobson radical of
$\overline{\pH(\pS)}$, that is, $[K,\overline{\pH(\pS)}]\subseteq \rad\bigl(\overline{\pH(\pS)}\bigr)$.
\end{proposition}

\begin{proof}
Let $K\in \overline{\pA(\pS)}$ be a compact operator and let $A\in \overline{\pH(\pS)}$ be an arbitrary operator.
In order to prove that the commutator $[K,A]$ is in $\rad\bigl(\overline{\pH(\pS)}\bigr)$, we need to show that
$[K,A]C$ is quasinilpotent for each $C\in \overline{\pH(\pS)}$. Since $[K,A]C$ is compact, the continuity of the spectral radius at compact operators yields that it suffices to prove that $[K,A]C$ is quasinilpotent for all $A, C\in \pH(\pS)$.

Let $A,C\in \pH(\pS)$ be arbitrary. Since $K$ belongs to the closure of $\pA(\pS)$ there is a sequence $(K_n)_{n\in \bN}\subseteq \pA(\pS)$ which converges to $K$. We claim that for each $n\in \bN$ the commutator $[K_n,A]$ is nilpotent and belongs to the ideal $\pN(\pS)_{_{\pH}}$ in $\pH(\pS)$ generated by the Jacobson
radical $\pN(\pS)$ of the band algebra $\pA(\pS)$. Since $A\in\pH(\pS)$, there exist operators
$A_1,\ldots, A_k\in \pA(\pS)$ and $B_0,\ldots,B_k\in \pA(\pS)'$ such that $A=B_0+\sum_{i=1}^kA_iB_i$.
It follows that $[K_n,A]=\sum_{i=1}^k[K_n,A_i]B_i.$ Since  every commutator $[K_n,A_i]$ is in
$\pN(\pS)$ and $\pN(\pS)\subseteq \pN(\pS)_{_{\pH}}$ it follows $[K_n,A_i] \in \pN(\pS)_{_{\pH}}$. Since
$\pN(\pS)_{_{\pH}}$ is an ideal in the hyperalgebra $\pH(\pS)$ we have $[K_n,A]C\in \pN(\pS)_{_{\pH}}$.
By \Cref{prop02}, the operator $[K_n,A]C$ is nilpotent. To finish the proof we apply the fact that the spectral radius is continuous at compact operators and that the compact operator $[K,A]C$ is the limit of the sequence $([K_n,A]C)_{n\in \bN}$ of nilpotent operators.
\end{proof}

\begin{theorem} \label{theo05}
Let $\pS\subseteq \pB(\eX)$ be an operator band. If there exists a non-zero compact operator in $\overline{\pA(\pS)}$, then $\pS$ has a non-trivial hyperinvariant subspace.
\end{theorem}

\begin{proof}
Let $K$ be a non-zero compact operator in $\overline{\pA(\pS)}$. Then, by \Cref{prop03}, $[K,A]$ is a compact operator contained in the radical $\rad\bigl(\pH(\pS)\bigr)$ for each $A\in \overline{\pH(\pS)}$. If $[K,A]=0$ for
every $A\in \overline{\pH(\pS)}$, then $\overline{\pH(\pS)}\subseteq (K)'$ and therefore every non-trivial hyperinvariant subspace of $K$ is a non-trivial hyperinvariant subspace for $\pA(\pS)$. Hence, we may assume that for some $A\in \overline{\pH(\pS)}$ the commutator $[K,A]$ is a non-zero operator. Let $\pJ$ be the ideal in $\overline{\pH(\pS)}$ generated by $[K,A]$. Of course, each operator in $\pJ$ is compact.
Since, by \Cref{prop03}, $[K,A]$ is in the Jacobson ideal of $\rad\bigl(\pH(\pS)\bigr)$ we have
$\pJ=\pH(\pS) [K,A] \pH(\pS)\subseteq \pH(\pS) \rad\bigl(\pH(\pS)\bigr) \pH(\pS)\subseteq \rad\bigl(\pH(\pS)\bigr)$.
It follows that operators in $\pJ$ are quasinilpotent. Thus, $\pJ$ is a Volterra ideal in $\overline{\pH(\pS)}$.
By \cite[Theorem 2]{Shu}, $\pJ$ is reducible. Now we apply \cite[ Lemma 7.4.6]{RR} and conclude that $\overline{\pH(\pS)}$ is reducible.
\end{proof}

\begin{corollary} \label{cor03}
Every essentially commuting non-scalar operator band $\pS\subseteq \pB(\eX)$ has a non-trivial hyperinvariant subspace.
\end{corollary}

\begin{proof}
If $\pS$ is commutative, then $\pS\subseteq \pS'$. In this case the kernel of any non-scalar operator from $\pS$ is invariant for each $S\in \pS'$. If $\pS$ is not commutative, then there exist idempotents $E,F\in \pS$ with a non-zero compact commutator $EF-FE\in \pA(\pS)$. The assertion follows, by \Cref{theo05}.
\end{proof}

Since every essentially commuting band of operators has an invariant subspace, an application of the Triangularization lemma (see \cite[Lemma 7.1.11]{RR}) immediately implies the following result.

\begin{corollary} \label{cor04}
Every essentially commuting band of operators on a Banach space is triangularizable.
\end{corollary}

{\it Acknowledgments.}
The paper is a part of the project {\em Distinguished subspaces of a linear operator} and the work of the first author
was partially supported by the Slovenian Research Agency through the research program P2-0268.
The second author acknowledges financial support from the Slovenian Research Agency, Grants No. P1-0222, J1-2453 and J1-2454.

%

%
\end{document}